\documentclass[12pt,reqno]{amsart}

\usepackage{times}
\usepackage[T1]{fontenc}
\usepackage{mathrsfs}
\usepackage{latexsym}
\usepackage[dvips]{graphics}
\usepackage{epsfig}
\usepackage{multicol}
\usepackage{amsmath,amsfonts,amsthm,amssymb,amscd}
\input amssym.def
\input amssym.tex
\usepackage{color}

\newcommand\be{\begin{equation}}
\newcommand\ee{\end{equation}}
\newcommand\bea{\begin{eqnarray}}
\newcommand\eea{\end{eqnarray}}
\newcommand\bi{\begin{itemize}}
\newcommand\ei{\end{itemize}}
\newcommand\ben{\begin{enumerate}}
\newcommand\een{\end{enumerate}}

\newtheorem{thm}{Theorem}

\newtheorem{cor}[thm]{Corollary}

\newtheorem{prop}[thm]{Proposition}
\newtheorem{exa}[thm]{Example}

\theoremstyle{definition}
\newtheorem*{remark}{Remark}

\newcommand{\Z}{\ensuremath{\mathbb{Z}}}
\newcommand{\Q}{\mathbb{Q}}

\newcommand{\p}{\mathfrak{p}}

  \title{Newly reducible iterates in families of quadratic polynomials}
  \author[Chamberlin, Colbert, Frechette, Hefferman, Jones, and Orchard]{KATHARINE CHAMBERLIN, EMMA COLBERT, Sharon Frechette, PATRICK HEFFERMAN, Rafe Jones, and SARAH ORCHARD}
  \thanks{This research was partially supported by a supplement to NSF grant DMS-0852826. All the authors are grateful for this support.}
  
\begin{document}
  
\begin{abstract}
We examine the question of when a quadratic polynomial $f(x)$ defined over a number field $K$ can have a newly reducible $n$th iterate, that is, $f^n(x)$ irreducible over $K$ but $f^{n+1}(x)$ reducible over $K$, where $f^n$ denotes the $n$th iterate of $f$. For each choice of critical point $\gamma$ of $f(x)$, we consider the family
$$
g_{\gamma,m}(x)=(x-\gamma)^2+m+\gamma, \qquad m \in K.
$$
For fixed $n \geq 3$ and nearly all values of $\gamma$, we show that there are only finitely many $m$ such that $g_{\gamma,m}$ has a newly irreducible $n$th iterate. For $n = 2$ we show a similar result for a much more restricted set of $\gamma$. These results complement those obtained by Danielson and Fein \cite{fein} in the higher-degree case. Our method involves translating the problem to one of finding rational points on certain hyperelliptic curves, determining the genus of these curves, and applying Faltings' theorem.
\end{abstract}

  \maketitle

\section{Introduction}

Let $K$ be a number field and $f(x) \in K[x]$. By the $n$th iterate $f^n(x)$ of $f(x)$, we mean the $n$-fold composition of $f$ with itself. Determining the factorization of $f^n(x)$ into irreducible polynomials has proven to be an important problem. 
%attracting the attention of many authors. 
From a dynamical perspective, it is a question about the inverse orbit $O^-(z) := \bigcup_{n \geq 1} f^{-n}(0)$ of zero, which has significance in various ways. For instance, it accumulates at every point of the Julia set of $f$ \cite[p. 71]{beardon}. The field of arithmetic dynamics seeks to understand sets such as $O^-(z)$ 
%this 
from an algebraic perspective, and finding the factorizations of $f^n(x)$ fits into this scheme: a nontrivial factorization arises from an ``unexpected" algebraic relation among elements of $O^-(z)$. In addition, understanding the factorization of $f^n(x)$ has proven to be a key obstacle in determining the Galois groups of $f^n(x)$ (see \cite{hjm, quaddiv} or \cite{galrat} for the case of some rational functions). These Galois groups provide a sort of dynamical analogue to the well-studied $\ell$-adic Galois representations \cite{arboreal}. 

In general, the factorization of the iterates of $f$ can exhibit a wide variety of behaviors. For instance, in \cite[Lemma 1.1]{fein-schacher} it is shown that for each $n \geq 1$ there exist (many) number fields $K$ such that for some $f(x) \in K[x]$, $f^{n+1}(x)$ is \textit{newly reducible}, that is, $f^n(x)$ is irreducible over $K$ but $f^{n+1}(x)$ is reducible over $K$. More specifically, it follows from \cite[p. 243]{stoll} and \cite[Lemma 1.1]{fein-schacher} that if $f(x) = x^2 + m$ for $m \in \Z_{> 0}$, $m \equiv 1,2 \bmod{4}$, then for any fixed $n \geq 1$ there exists a number field $K$ such that $f^{n+1}(x)$ is newly 
%irreducible 
reducible over $K$. But what happens when we fix the number field $K$ to start with, and ask about the factorization of $f^n(x)$ as $n$ grows? Many authors have examined this question, in general with the aim of giving criteria that ensure all iterates are irreducible (see, e.g.,\cite{itconst}\cite[Section 4]{odoniwn}). Most usefully for our purposes, Danielson and Fein \cite{fein} examine the case when $f(x) = x^d + m$, for $d \geq 2$. They show, for instance, that if $m \in \Z$ and $f(x)$ is irreducible, then all iterates of $f$ are irreducible. In fact they only assume that 
%$R$ is a UFD with field of fractions $K$
$K$ is the quotient field of a unique factorization domain $R$, and in this case they show that certain strong diophantine conditions must be satisfied when $f^n(x)$ is irreducible and $f^{n+1}(x)$ is reducible. In particular, for $K = \Q$, they take $S(d,n)$ to be the set of $m \in \Q$ such that $f^{n+1}(x)$ is newly reducible.
%, where $f(x) = x^d + m$. 
Further, let $S(d) = \bigcup_{n \geq 1} S(d,n)$. In \cite[Theorem 7]{fein} it is shown, among other things, that $S(2,1)$ (and thus $S(2)$) is infinite, $S(3,n)$ is finite for all $n \geq 1$, and $S(d)$ is finite for $d$ odd, $d \geq 5$. Moreover, the abc-conjecture implies that $S(d)$ is finite for $d$ even, $d \geq 4$. 

One goal of the present paper is to determine whether $S(2,n)$ is finite for $n \geq 2$. Our main result, however, is significantly more general. 
Consider the family of polynomials
\begin{equation} \label{thefamily}
g_{\gamma,m}(x)=(x-\gamma)^2+m+\gamma, \qquad \gamma,m \in K,
\end{equation} 
where $K$ is a number field.  Denote the ring of integers of $K$ by $\mathcal{O}_K$.
Our main result is the following:
%Our principal question is the following: for fixed $m$, what is the set of $m \in K$ for which $g_{\gamma,m}^n(x)$ is irreducible and $g_{\gamma,m}^{n+1}(x)$ is reducible? The precise answer to this question depends on $n, \gamma$ and, crucially, $K$. However, we find the set to be finite in a wide variety of circumstances:

\begin{thm} \label{smores bites}
Let $K$ be a number field, $v_\p$ the valuation attached to a prime $\p$ of $\mathcal{O}_K$, and $g_{\gamma,m}(x)$ as in \eqref{thefamily}. If one of the following hold, then there are only finitely many $m$ such that $g_{\gamma,m}^{n}(x)$ is irreducible and $g_{\gamma,m}^{n+1}(x)$ is reducible:
\begin{enumerate}
\item $n \ge 3$ and there exists a prime $\p$ of $\mathcal{O}_K$ with $v_\p(2) = e \geq 1$ and $v_\p(\gamma) = s$ with $s \neq -e2^i$ for all $i \geq 1$. 
%\item $n\ge 3$ and $\gamma=\frac{r}{s}$ with $r,s\in \mathbb{Q}$ and $\text{gcd}(r,s)=1$ where $4\nmid s$.
\item $n=2$ and $\gamma = r/4$ for $-200 \leq r \leq 200$.
\end{enumerate}
\end{thm}

%\begin{defi} \cite{fein} Let $S(n,m)$ be the set of $b\in K$ such that, for $g_{\gamma,b}(x)=(x-\gamma)^n+b+\gamma$, $g_{\gamma,b}^m(x)$ is irreducible in $K[x]$ but $g_{\gamma,b}^{m+1}(x)$ is reducible over $K$. \end{defi}

%So for the purposes of this paper, we are examining $S(2,m)$ where $m\ge 2$. An infinite number of examples of irreducibility of the first iterate and reducibility of the second can be found, as will be shown later in this paper. This conclusion  leads us to the question: is it possible that for higher iterates, namely $n\ge2$, $g_{\gamma,m}^{n+1}(x)$ is reducible when $g_{\gamma,m}^n(x)$ is irreducible?

In particular, when $K = \Q$, part (1) of Theorem \ref{smores bites} holds when $v_2(\gamma)$ is not of the form $-2^j$ for $j \geq 1$. Hence when $\gamma = 0$, we obtain that $S(2,n)$ is finite for $n \geq 2$ (in the notation of \cite{fein}); in other words, for each $n \geq 2$ there are at most finitely many $m \in \Q$ such that $x^2 + m$ has a newly reducible $(n+1)$st iterate. In Proposition \ref{deep dish pizza}, we show further that $S(2,3)$ is empty. Note also that part (1) of Theorem \ref{smores bites} applies whenever $\gamma$ belongs to the ring of integers of $K$, and in particular for $\gamma \in \Z$. In fact, part (1) holds whenever $\gamma$ is taken so that:
\begin{equation} \label{nonsingcond}
\text{$g_{\gamma,m}^i(\gamma)  \in K[m]$ does not have repeated roots for any $i \geq 1$.} 
\end{equation}
(See Theorem \ref{baked ziti}, Proposition \ref{twizzler}, and the discussion immediately preceding Proposition \ref{twizzler}.) Condition \eqref{nonsingcond} is the same as the condition appearing in \cite{preimagecurve} for the \textit{preimage curve} $Y^{\textrm{pre}}(i,-\gamma)$, given by the vanishing of the polynomial $(g_{0,m}^i(x) + \gamma) \in K[x,m]$, to be non-singular for all $i \geq 1$. In Proposition \ref{twizzler}, we give a new criterion ensuring \eqref{nonsingcond} holds for given $\gamma$, thereby improving \cite[Proposition 4.8]{preimagecurve}. It seems reasonable to believe that part (1) of Theorem \ref{smores bites} holds even when condition \eqref{nonsingcond} fails; see the remark following the proof of Proposition \ref{twizzler}.

There are infinitely many $m$ such that $g_{\gamma,m}(x)$ is irreducible and $g_{\gamma,m}^{2}(x)$ is reducible, and one can explicitly describe them (see Theorem \ref{macNcheese}), thus settling the $n=1$ case. When $\gamma = 0$ this result is \cite[Proposition 2]{fein}.  For given $K$, let us denote by $S(2,n,\gamma)$ the set of $m \in K$ such that $g_{\gamma,m}^{n+1}(x)$ is newly reducible.
We note that that even when $K = \Q$, the sets $S(2,n,\gamma)$ may be non-empty. For instance, when $f(x) = x^2 - x - 1$, corresponding to $\gamma = 1/2$ and $m = -7/4$, we have that $f(x)$ and $f^2(x)$ are irreducible but
\begin{equation} \label{n3example}
f^3(x)=(x^4 - 3x^3 + 4x - 1)(x^4 - x^3 - 3x^2 + x + 1),
\end{equation}
and thus $-7/4 \in S(2,2,1/2)$. For $K = \Q$, the sets $S(2,n,\gamma)$ are likely to be empty for $n \geq 3$, since as we will see they correspond to rational points on high-genus curves. However, without effective algorithms to find such points, a new approach will be required to precisely determine $S(2,n,\gamma)$. 

To prove Theorem \ref{smores bites}, we first examine the case where $n\ge 3$ and use the fact that comparing constant terms of a hypothetical non-trivial factorization of $g_{\gamma,m}^{n+1}(x)$ gives rise to $K$-rational points on a hyperelliptic curve (at least for the $\gamma$ described in part (1) of Theorem \ref{smores bites}). This allows us to use Faltings' Theorem to conclude that $S(2,n,\gamma)$ is finite for these $\gamma$ and for $n \ge 3$. We then examine the $n=2$ case using a system of equations generated from a factorization of the third iterate. After defining certain cases for this system, we use Faltings' Theorem on a plane curve arising from the Groebner basis of the system to show that $S(2,2, \gamma)$ is finite for certain $\gamma$. 
%In particular, we will be examining the two cases where $\gamma=0$, which we will refer to as $f_q(x)$, and $\gamma=\frac{1}{2}$. Combining the $n\ge 3$ and $n=2$ case will allow us to prove Theorem \ref{smores bites}.

\section{The $n=1$ Case}

Before we approach the main theorem, let's examine the case where $n=1$. It is possible for $g_{\gamma,m}^2(x)$ to be reducible and $g_{\gamma,m}(x)$ irreducible: 
\begin{exa} \label{firstex}
Let $\gamma=0$, $m=-\frac{4}{3}$, and $K=\mathbb{Q}$. Then $$g_{0,-\frac{4}{3}}(x)=x^2-\frac{4}{3}$$ is irreducible over $\mathbb{Q}$ since $\frac{4}{3}$ is not a rational square. However, we have $$g_{0,-\frac{4}{3}}^2(x)=\Bigg(x^2-\frac{4}{3}\Bigg)^2-\frac{4}{3}=\Bigg(x^2-2x+\frac{2}{3}\Bigg)\Bigg(x^2+2x+\frac{2}{3}\Bigg).$$ 
\end{exa}

%Danielson and Fein determined the possible factorizations of $f_q^2(x)$. 
Because it has degree $4$, $g_{\gamma,m}^2(x)$ could a priori have non-trivial factors of degree 1, 2, or 3. We will show in Corollary \ref{pulled pork} that if $g_{\gamma,m}(x)$ is irreducible, then the only non-trivial factorization for $g_{\gamma,m}^2(x)$ is $p_1(x)p_2(x)$ with $\deg p_1(x)=\deg p_2(x)=2$.

\begin{thm} \label{macNcheese} We have $g_{\gamma,m}(x)$ irreducible and $g_{\gamma,m}^2(x)$ reducible if and only if either
\begin{enumerate} 
\item $\gamma \neq 1/4$ and ${\displaystyle m=\frac{c_1^4-4\gamma}{4-4c_1^2}}$, where  $c_1\in K \setminus \{-1,1\}$ and ${\displaystyle \frac{4\gamma-c_1^2}{1-c_1^2}} \notin K^{*2}$; or 
\vspace{0.1 in}
\item $\gamma = 1/4$ and $-4m - 1 \notin K^{*2}$.
\end{enumerate}
In particular, for each $\gamma \in K$, the set $S(2,1,\gamma)$ is infinite.%there are infinitely many $m \in K$ such that $g_{\gamma,m}^2(x)$ is newly reducible. 
\end{thm} 
 
\begin{remark}
It is interesting to note that when $\gamma = 1/4$, we have 
\begin{equation} \label{badfamily}
g^2_{1/4,m}(x) = \left(x^2 - \frac{3}{2}x+ (m + 13/16) \right) \left(x^2 + \frac{1}{2}x + (m + 5/16) \right),
\end{equation}
%% changed z to x on the right-hand side of the above equation
and so $g^2_{1/4,m}(x)$ is reducible for all $m \in K$. This phenomenon has already been noticed, albeit in somewhat different language, in \cite[Remark 2.6 and p. 94]{preimagecurve}.
\end{remark} 
 
\begin{proof}
 Suppose that $g_{\gamma,m}(x)$ is irreducible and $g_{\gamma,m}^2(x)$ is reducible, so that $g_{\gamma,m}^2(x)=p_1(x)p_2(x)$. Write $p_1(x)=(x-\gamma)^2+b_1(x-\gamma)+b_0$ and $p_2(x)=(x-\gamma)^2+c_1(x-\gamma)+c_0$, where $b_i,c_i\in K$, and note that
\begin{equation}
g_{\gamma,m}^2(x)=(x-\gamma)^4+2m(x-\gamma)^2+m^2+m+\gamma.
\end{equation}
Comparing coefficients in the equality $g_{\gamma,m}^2(x)=p_1(x)p_2(x)$ gives the following system of equations:
\begin{multicols}{2}
\begin{enumerate}
\item $c_1+b_1=0$
\item $c_0+b_1c_1+b_0=2m$
\item $b_1c_0+b_0c_1=0$
\item $b_0c_0=m^2+m+\gamma.$
\end{enumerate}
\end{multicols}
Clearly $b_1=-c_1$ from (1). We show $b_0 = c_0$ in Theorem \ref{baked ziti} (see \eqref{consteq}). So we have the following system of two equations:
\begin{enumerate} 
\item[(a)] $2c_0-c_1^2-2m=0$
\item[(b)] $c_0^2-m^2-m-\gamma=0.$
\end{enumerate}
Solving (1) for $c_0$ and substituting the result into equation (2) gives
\begin{equation} \label{maineq}
c_1^4+4mc_1^2-4m-4\gamma = 0.
\end{equation}
If $c_1 = \pm 1$, then $\gamma = 1/4$. Thus in the case where $\gamma \neq 1/4$, we may solve \eqref{maineq} for $m$ to obtain $m=(c_1^4-4\gamma)/(4-4c_1^2).$ Because $g_{\gamma,m}(x)$ is assumed to be irreducible, we clearly have from \eqref{thefamily} that $-m - \gamma \notin K^{*2}$, and one computes $-m - \gamma = (c_1^2(4\gamma-c_1^2))/(4(1-c_1^2))$. In the case where $\gamma = 1/4$, we may take 
%$c_0 = \pm 1$
$c_1 = \pm 1$ and $c_0 = (1 + 2m)/2$ to get a solution to equations (a) and (b) (this is the same as the factorization in \eqref{badfamily}). Hence $g^2_{1/4,m}(x)$ is reducible for all $m \in K$. Clearly $-m - \gamma = -m - 1/4$, which is in $K^{*2}$ if and only if $-4m - 1$ is. 
%So $m\in K$ provided $c_1
%\in K-\{-1,1\}$. Now we must find a further condition on $c_1$ to ensure that $g_{\gamma,m}(x)$ is irreducible. First we must examine the roots of $g_{\gamma,m}(x)$:
%\begin{eqnarray}
%(x-\gamma)^2+m+\gamma & = & 0 \nonumber\\
%(x-\gamma)^2 & = & -m-\gamma \nonumber\\
%x & = & \gamma \pm \sqrt{-m-\gamma}. \nonumber
%\end{eqnarray}
%These roots are in $K$ whenever $-\gamma-m \in K^2$; in other words, $g_{\gamma,m}(x)$ is reducible whenever $\gamma \pm \sqrt{-\gamma-m} \in K.$ Using Equation \ref{pizza}, we know \begin{eqnarray} 
%-\gamma-m & = & -\gamma -\frac{c_1^4-4\gamma}{4-4c_1^2} \nonumber\\
%& = & \frac{-\gamma(4-4c_1^2)-c_a^4+4\gamma}{4-4c_1^2} \nonumber\\
%& = & \frac{4c_1^2\gamma-c_1^4}{4-4c_1^2} \nonumber\\
%& = & \frac{c_1^2(4\gamma-c_1^2)}{4(1-c_1^2)} \nonumber \end{eqnarray}
%\noindent We know that $\frac{c_1^2}{4} \in \mathbb{Q}^2$, so $g_{\gamma,m}(x)$ is reducible whenever $\frac{4\gamma-c_1^2}{1-c_1^2} \in K^2$.

Assume now that either (1) or (2) holds. Then $-m - \gamma$ is not in $K^{*2}$, so $g_{\gamma,m}(x)$ is irreducible. The other hypotheses ensure that equations (a) and (b) above have solutions in $K$, and hence $g^2_{\gamma,m}(x)$ is reducible. 
\end{proof}

Note that when $\gamma = 0$, taking $c_1 = 2$ in Theorem \ref{macNcheese} yields Example \ref{firstex}. We also remark that in the case of $\gamma = 0$, taking $c_1 = 2z$ in Theorem \ref{macNcheese} yields Proposition 2 of \cite{fein}, at least in the case where $K$ is a number field. (Note that in \cite{fein} the polynomial under consideration is $x^2 - m$, and hence the results differ by a minus sign.)

\section{The $n\ge3$ Case}

To prove Theorem \ref{smores bites}, there are multiple issues we must address, and thus we prove it in several steps. 
%We proved the $n=1$ case first because it was the most straightforward. Recall that for the $n=1$ case, by solving a system of equations extracted from the second iterate, we were able to find the conditions on $\gamma, m,$ and $c_1$  that proved Theorem \ref{smores bites}. Next, 
Having handled the $n = 1$ case, we now address the case where $n\ge3$ (rather than $n=2$, which is the subject of Section \ref{twocase}) because the former allows us to work with hyperelliptic curve of genus at least two, while the latter does not. Therefore we will be able to directly prove the $n\ge3$ case, while additional work is involved in proving the $n=2$ case. 

Understanding the roots of $g_{\gamma,m}^{n+1}(x)$ is important to proving some key results. 
%The roots of $g_{\gamma,m}(x)$ are $\gamma\pm\sqrt{-m-\gamma}$, and the roots of $g_{\gamma,m}^2(x)$ are $\gamma\pm\sqrt{-m-\sqrt{-m-\gamma}}$. We generalize these roots for $g_{\gamma,m}^n(x)$ and $g_{\gamma,m}^{n+1}(x)$. 
In general, if $\beta_i$ is a root of $g_{\gamma,m}^n(x)$, then the two roots of $g_{\gamma,m}(x) - \beta_i$ are roots of $g_{\gamma,m}^{n+1}(x)$. Calling them $\alpha_i^+$ and $\alpha_i^-$, we have $\alpha_i^+=\gamma+\sqrt{\beta_i-m-\gamma}$ and $\alpha_i^-=\gamma-\sqrt{\beta_i-m-\gamma}$. Note that $$2\gamma-\alpha_i^+=2\gamma-(\gamma+\sqrt{\beta_i-m-\gamma})=\gamma-\sqrt{\beta_i-m-\gamma}=\alpha_i^-.$$ The following picture summarizes the relation of the roots to one another. Note that they are arranged in a tree.

\begin{figure}[htp]
	\resizebox{\textwidth}{!}{
		\includegraphics[width=50pt]{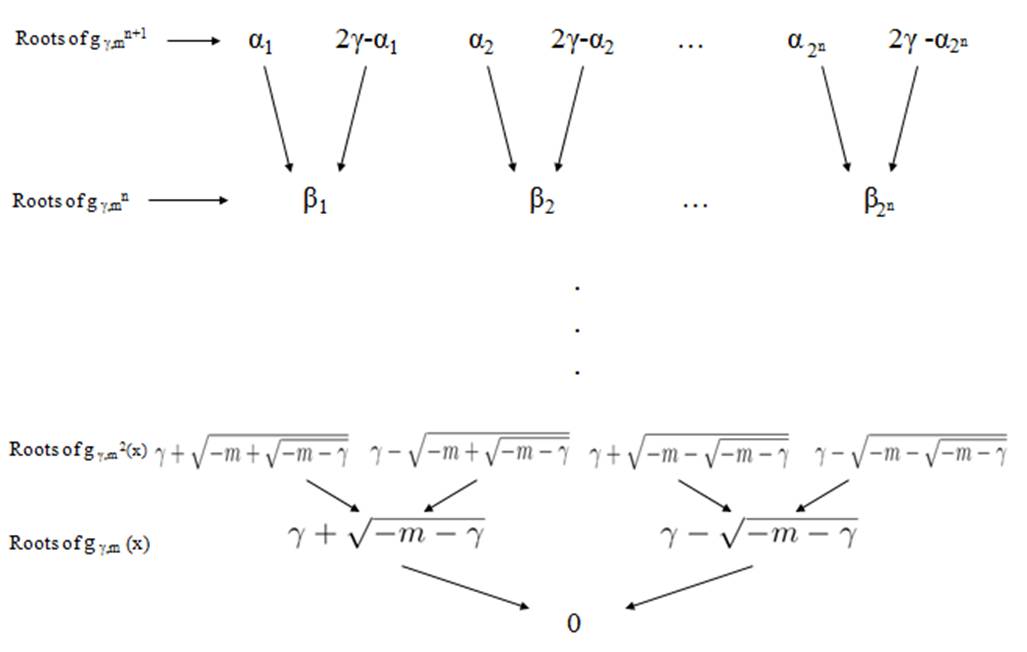} 
		}
	\end{figure}

In this section we establish two principal results on the structure of hypothetical factors in the case where $g_{\gamma,m}^{n+1}(x)$ is newly reducible. Our first result is similar to \cite[Proposition 2.6]{settled}.

\begin{thm} \label{pumpkin bread} Let $g_{\gamma,m}(x)=(x-\gamma)^2+\gamma+m$, as above.  Suppose $g_{\gamma,m}^n(x)$ is irreducible, and $g_{\gamma,m}^{n+1}(x)=p_1(x)p_2(x)$ where $p_1(x)$ and $p_2(x)$ are non-trivial factors. If $\alpha$ is a root of $p_1(x)$, then $2\gamma-\alpha$ is not a root of $p_1(x)$. \end{thm}

\begin{proof}
We proceed by contradiction. Assume there is a root $\alpha$ of $p_1(x)$ such that $p_1(\alpha)=p_1(2\gamma-\alpha)=0$. Let $E_{n+1}$ be the splitting field of $g_{\gamma,m}^{n+1}(x)$ over $K$ and $E_n$ be the splitting field of $g_{\gamma,m}^n(x)$ over $K$. Then $K\subset E_n\subset E_{n+1}$. Consider the Galois groups $G_n=\mathrm{Gal}(E_n/K)$ and $G_{n+1}=\mathrm{Gal}(E_{n+1}/K)$. The action of $G_{n+1}$ on the roots of $g_{\gamma,m}^n(x)$ is the same as the action of $G_n$ on the roots of $g_{\gamma,m}^n(x)$; in other words, $G_{n+1}$ acts transitively on the roots of $g_{\gamma,m}^n(x)$ because $g_{\gamma,m}^n(x)$ is irreducible over $K$.

\textit{Claim:}  If $\alpha'$ is any root of $g_{\gamma,m}^{n+1}(x)$, then $p_1(\alpha')=0$.
Let $\phi\in G_{n+1}$ such that $\phi(g_{\gamma,m}(\alpha))=g_{\gamma,m}(\alpha')$ since $G_{n+1}$ acts transitively on roots of $g_{\gamma,m}^n(x)$. Then,
\begin{eqnarray}
\phi(g_{\gamma,m}(\alpha)) & = & g_{\gamma,m}(\alpha') \nonumber\\
\phi((\alpha-\gamma)^2+\gamma+m ) & = & (\alpha'-\gamma)^2+\gamma+m \nonumber\\
\phi(\alpha-\gamma)^2+\gamma+m & = & (\alpha'-\gamma)^2+\gamma+m \nonumber\\
\phi(\alpha-\gamma) & = & \pm (\alpha'-\gamma) \nonumber\\
\phi(\alpha)-\gamma & = & \pm (\alpha'-\gamma) \nonumber\\
\phi(\alpha) & = & \gamma\pm (\alpha'-\gamma) \nonumber.
\end{eqnarray}
Thus either $\phi(\alpha) = \gamma+(\alpha'-\gamma) = \alpha'$ or $\phi(\alpha) = \gamma-(\alpha'-\gamma) = 2\gamma - \alpha'$. 
%If $\phi(\alpha) = \gamma+(\alpha'-\gamma)$, then $\phi(\alpha)=\alpha'$. If $\phi(\alpha) = \gamma-(\alpha'-\gamma)$, then $\phi(\alpha)=2\gamma-\alpha'$.

%\noindent So for our root $\alpha$, there exists a $\phi \in G_{n+1}$ such that either $\phi(\alpha)=\alpha'$ or $\phi(\alpha)=2\gamma-\alpha'$. 
\noindent If $\phi(\alpha)=\alpha'$, then $$\phi(2\gamma-\alpha)=2\gamma-\phi(\alpha)=2\gamma-\alpha'.$$ On the other hand, if $\phi(\alpha)=2\gamma-\alpha'$, then $$\phi(2\gamma-\alpha)=2\gamma-\phi(\alpha)=2\gamma-(2\gamma-\alpha')=\alpha'.$$ Therefore, $\alpha$ and $2\gamma-\alpha$ collectively map to all roots $\alpha_i$ of $g_{\gamma,m}^{n+1}(x)$. Since all roots of $p_1(x)$ must be mapped to other roots of $p_1(x)$ under $\phi$, all $\alpha_i$ are roots of $p_1(x)$.  Consequently $p_1(x)$ has $2^{n+1}$ roots and degree $2^{n+1}$. 
%So $p_1(x)$ has $2^{n+1}$ roots, which implies that $\deg p_1(x)=2^{n+1}$. 
Since $\deg p_1(x) + \deg p_2(x)= \deg g_{\gamma,m}^{n+1}(x)= 2^{n+1}$, we have $\deg p_2(x)=0$, contradictiing the hypothesis that $p_1(x)$ and $p_2(x)$ are non-trivial factors.
\end{proof}

\begin{cor} \label{pulled pork}
Let $g_{\gamma,m}(x)=(x-\gamma)^2+m+\gamma$ with $\gamma, m \in K$. Let $n \in \mathbb{Z}^+$, and assume $g_{\gamma,m}^n(x)$ is irreducible with $g_{\gamma,m}^{n+1}(x)=p_1(x)p_2(x)$, where $p_1(x)$ and $p_2(x)$ are non-trivial factors. Then, $\deg p_1(x) = \deg p_2(x) = 2^n$, and $p_1(x)$ and $p_2(x)$ are irreducible.
\end{cor}

\begin{proof} Assume $g_{\gamma,m}^n(x)$ is irreducible and $g_{\gamma,m}^{n+1}(x)= p_1(x)p_2(x)$. We know 
%that 
$\deg g_{\gamma,m}^n(x)$ $=2^n$ and $\deg g_{\gamma,m}^{n+1}(x)=2^{n+1}$. Let $\{\alpha_1,...,\alpha_{2^n}\}$ be the roots of $p_1(x)$ where $\alpha_j\not=2\gamma-\alpha_i$ for $i\not=j$ and $i,j\in\{1,\ldots,2^n\}$. Without loss of generality, let $\alpha_i=\gamma+\sqrt{\beta_i-m-\gamma}$. Also consider $ G_{n+1}$ as defined in Theorem \ref{pumpkin bread}. Let $\alpha$ be a root of $g_{\gamma,m}^{n+1}(x)$ and $\phi\in G_{n+1}$. If $\phi(\alpha)=\alpha'$, then $\alpha'$ and $2\gamma-\alpha'$ cannot be in the same orbit of $G_{n+1}$ over the roots of $g_{\gamma,m}^{n+1}(x)$ by Theorem \ref{pumpkin bread}. Thus each orbit of $G_{n+1}$ over the roots of $g_{\gamma,m}^{n+1}(x)$ has no more than $2^n$ elements.

But, since $G_{n+1}$ acts transitively on $\{\beta_1,\ldots,\beta_{2^n}\}$, there is only one orbit of $G_{n+1}$ over $\{\beta_1,\ldots,\beta_{2^n}\}$. Let $\beta_i$ be a root of $g_{\gamma,m}^n(x)$ with $1\le i \le 2^n$. Then,
\begin{eqnarray}
\mathrm{orb}_{G_{n+1}}(\beta_i) & = & \{\beta_1,\ldots,\beta_{2^n}\} \nonumber\\
\implies \mathrm{orb}_{G_{n+1}}((\alpha_i-\gamma)^2+\gamma+m) & = & \{\beta_1,\ldots,\beta_{2^n}\}. \nonumber
\end{eqnarray}
Therefore, each orbit of $G_{n+1}$ over the roots of $g_{\gamma,m}^{n+1}(x)$ has at least $2^n$ elements. Hence each orbit of $G_{n+1}$ over the roots of $g_{\gamma,m}^{n+1}(x)$ has exactly $2^n$ elements so $g_{\gamma,m}^{n+1}$ has two irreducible factors, each with degree $2^n$. \end{proof}

\section{Curves and Faltings' Theorem}
We now use Theorem \ref{pumpkin bread} to connect a new factorization of $g_{\gamma,m}^{n+1}(x)$ to $K$-rational points on a certain curve. 

\begin{thm} \label{baked ziti}
If $g_{\gamma,m}^n(x)$ is irreducible and $g_{\gamma,m}^{n+1}(x)$ is reducible for some $n\ge 1$, then there exist $x,y\in K$ with $x=m$ such that $$y^2=t_{n+1}(x),$$ where the polynomials $t_i(x)$ are defined by the recurrence relation $t_1(x)=x+\gamma$ and for $i\ge 2$, $$t_i(x)=(t_{i-1}(x)-\gamma)^2+x+\gamma.$$
\end{thm}

\begin{remark}
Note that $t_i(x) = (g_{\gamma,m}^i(\gamma))|_{m = x}$, as will be shown below (or can be easily seen by induction). 
\end{remark}

\begin{proof} Assume $g_{\gamma,m}$ is irreducible and $g_{\gamma,m}^{n+1}(x)=p_1(x)p_2(x)$ for some $p_1(x), p_2(x) \in K[x]$ of positive degree. By Theorem \ref{pumpkin bread}, if $\alpha_i$ is a root of $p_1(x)$ for any $i \in \{1,\cdots ,2^n\}$, then $2\gamma-\alpha_i$ is not a root of $p_1(x)$, which implies that $2\gamma-\alpha_i$ is a root of $p_2(x)$. Without loss of generality, let $\alpha_1,\ldots,\alpha_{2^n}$ be the roots of $p_1(x)$ with $\alpha_j\not=2\gamma-\alpha_i$ and $\alpha_i=\gamma+\sqrt{\beta_i-\gamma-m}$. Then,
\begin{eqnarray}
p_1(x) & = & (x-\alpha_1)(x-\alpha_2) \cdots (x-\alpha_{2^n}) \text{ and} \nonumber\\
p_2(x) & = & (x-(2\gamma-\alpha_1))(x-(2\gamma-\alpha_2)) \cdots (x-(2\gamma-\alpha_{2^n})) \nonumber\\
& = & (x-2\gamma+\alpha_1)(x-2\gamma+\alpha_2) \cdots (x-2\gamma+\alpha_{2^n}). \nonumber
\end{eqnarray}

\noindent So we have:
\begin{eqnarray}
p_1(\gamma) & = & (\gamma-\alpha_1)(\gamma-\alpha_2) \cdots (2\gamma-\alpha_{2^n}) \text{, and} \nonumber\\ 
p_2(\gamma) & = & (-\gamma+\alpha_1)(-\gamma+\alpha_2) \cdots (-\gamma+\alpha_{2^n}) \nonumber\\
& = & (-1)^{2^n} (\gamma-\alpha_1)(\gamma-\alpha_2) \cdots (\gamma-\alpha_{2^n}), \label{consteq}
\end{eqnarray}

\noindent and therefore $p_1(\gamma)=p_2(\gamma)$. Set $y=p_1(\gamma)=p_2(\gamma)$, so $g_{\gamma,m}^{n+1}(\gamma)=y^2$. We have
$$g_{\gamma,m}^{n+1}(x)= g(g_{\gamma,m}^n(x))=(g_{\gamma,m}^n(x)-\gamma)^2+m+\gamma,$$
and hence
%so the constant term of $g_{\gamma,m}^{n+1}(x)$ is
$$g_{\gamma,m}^{n+1}(\gamma)= g(g_{\gamma,m}^n(\gamma))=(g_{\gamma,m}^n(\gamma)-\gamma)^2+m+\gamma.$$
Moreover, $g_{\gamma,m}(\gamma)=m+\gamma$ and  $g_{\gamma,m}^i(\gamma)$ satisfies the same recurrence relation as $t_i(x)$, with $x$ replaced by $m$.
\end{proof}

The polynomials $t_i(x)$ play a critical role in our argument. The first few are:
\begin{equation}
t_1(x) = x + \gamma \qquad t_2(x) = x^2 + x + \gamma \qquad t_3(x) = x^4 + 2x^3 + x^2 + x + \gamma 
\end{equation}
$$t_4(x) = x^8 + 4x^7 + 6x^6 + 6x^5 + 5x^4 + 2x^3 + x^2 + x + \gamma $$
Equations of the form $y^2=t_i(x)$ may be interpreted geometrically as plane curves. A plane curve defined over a field $F$ is the set of solutions $(x,y)\in F\times F$ of an equation of the form $h(x,y)=0$, where $h(x,y)\in F[x,y]$. If $K$ is a subfield of $F$, a $K$-rational point on the curve is one whose coordinates lie in $K$. For instance, $(1,-1)$ is a $\Q$-rational point on the curve $y^2 = x^3 + x - 1$, while $(-1, \sqrt{-3})$ is not (though it is $K$-rational for $K = \Q(\sqrt{-3})$). 

The genus of a plane curve is a measure of its geometric complexity, and for curves of the form $y^2 = r(x)$, which is the case of interest to us in light of Theorem \ref{baked ziti}, there is a convenient way to calculate it -- at least, when the roots of $r(x)$ in the algebraic closure of $K$ are distinct. 
\begin{thm} \label{panera} \cite{Goldschmidt} Consider the curve $C:  y^2=r(x)$. If $r(x)$ is separable and of degree $d$, then the genus $g$ of $C$ is given by  
\begin{displaymath}
   g = \left\{
     \begin{array}{lr}
       (d-1)/2  & \textrm{ for } d \textrm{ odd},\\
       (d-2)/2  & \textrm{ for } d \textrm{ even}.\\
%        \frac{1}{2}(d-1)  & \textrm{ for } d \textrm{ odd},\\
%       \frac{1}{2}(d-2)  & \textrm{ for } d \textrm{ even}.\\
    \end{array}
   \right.
\end{displaymath}
\end{thm}

Assume that $r(x)$ is separable. A curve of the form $y^2 = r(x)$ of genus at least two is called a \textit{hyperelliptic curve}, while when such a curve has genus one it is known as a \textit{elliptic curve}.  The reason we care about the genus of a curve is that Faltings' Theorem famously connects it to the number of $K$-rational points on the curve:
%, which allows us to apply Faltings' Theorem to determine the set of $K$-rational points on a curve. If we can determine whether or not this set is finite for curves with $i\ge 3$, we will have proven Theorem \ref{smores bites} for $g\ge3$.
\begin{thm}\textbf{(Faltings' Theorem)} \cite{Hindry} Let $K$ be a number field, and let $C$ be a curve defined over $K$ of genus $g\ge2$. Then the set of $K$-rational points on  $C$ is finite.  \end{thm}

Suppose for a moment that all of the polynomials $t_i(x)$ in Theorem \ref{baked ziti} are separable. Clearly $\deg t_i(x) = 2^{i-1}$. By Theorem \ref{panera}, the genus $g_i$ of the curve $y^2 = t_i(x)$ then satisfies 
\begin{equation} \label{tigenus}
   g_i = \left\{
     \begin{array}{lr}
       0  & \textrm{for} \hspace{3mm} i=1,\\
       2^{i-2}-1 & \textrm{for} \hspace{3mm} i\ge2.\\
     \end{array}
   \right.
\end{equation}
Therefore by Faltings' Theorem, the curve $y^2 = t_{n+1}(x)$ has only finitely many $K$-rational points for $n \ge 3$. In particular, there are only finitely many $x \in K$ such that $(x, y)$ is a $K$-rational point on $y^2 = t_{n+1}(x)$. Thus, by Theorem \ref{baked ziti}, when $n\ge3$ there are only finitely many $m \in K$ with $g_{\gamma,m}^n(x)$ irreducible and $g_{\gamma,m}^{n+1}(x)$ reducible over $K$.
% for $\gamma=\frac{r}{s}$ where $r,s \in K$ with $gcd(r,s)=1$ and $s$ odd. 

Hence the lone remaining obstacle to proving part (1) of Theorem \ref{smores bites} is to establish that the $t_i(x)$ in Theorem \ref{baked ziti} are separable. Note that this is not true for all $\gamma \in K$. Indeed, if $\gamma = 1/4$, then $t_2(x) = (x + 1/2)^2$. (Note that $\gamma = 1/4$ also distinguished itself in 
%the 
Theorem \ref{macNcheese}). The set $S : = \{\gamma \in \overline{\Q} : \text{$t_i(x)$ is separable for all $i \geq 1$}\}$ is the same as the set of $a \in \overline{\Q}$ such that the pre-image curves $Y^{\textrm{pre}}(N,-a)_{N \geq 1}$ defined in \cite{preimagecurve} are all non-singular. In general, the set of $\overline{\Q} \setminus S$ is poorly understood. One result \cite[Proposition 4.8]{preimagecurve} gives a criterion for membership in $S$. Here we give an improvement on that result.

\begin{prop} \label{twizzler} Let $K$ be a number field with ring of integers $\mathcal{O}_K$, and let $t_i(x)=(t_{i-1}(x)-\gamma)^2+x+\gamma$,  as in Theorem \ref{baked ziti}.
%$\gamma=\frac{r}{s}$ where $r,s \in K$ with $gcd(r,s)=1$ and $s$ odd.
Suppose there exists a prime $\p$ of $\mathcal{O}_K$ with 
$v_\p(2) = e \geq 1$ and $v_\p(\gamma) = s$ with $s \neq -e2^j$ for all $j \geq 1$.
Then $t_i(x)$ is separable over $K$ for all $i\ge 1$. 
\end{prop}

\begin{remark}
When $K = \Q$, Proposition \ref{twizzler} says that if $v_2(\gamma) \neq -2^j$ for all $j \geq 1$, then $t_i(x)$ is separable for all $i \geq 1$.
%[Need to update this] Because Proposition \ref{twizzler} is false for $\gamma = 1/4$, we can't improve the hypotheses of Proposition to $v_\p(\gamma) \geq r$ for any $r < -1$. 
\end{remark}

\begin{proof}
It suffices to establish that $t_i(x)$ and $t_i'(x)$ have no common roots in 
%$\overline{\Q}$
$\overline{K}$, which we do through the use of Newton polygons with respect to the valuation $v_\p$ (which we often abbreviate NP). We assume the reader is familiar with the relationship between slopes of the Newton polygon of a polynomial and the $\p$-adic valuation of the polynomial's roots (see e.g. \cite[Theorem 5.11]{ADS}).
We first claim that for each $r$ with $0 \leq r \leq i-2$, $t_i'(x)$ has $2^r$ roots in $\overline{K}$ with $\p$-adic valuation $-e/2^{r}$. The statement is trivial for $i = 2$, so we assume inductively that it holds for given $i \geq 3$, and we consider the NP of $t_i'(x)$ with respect to the $\p$-adic valuation. By the chain rule,
\begin{equation*}
t'_{i+1}(x)=2(t_{i}(x)-\gamma)t'_{i}(x)+1.
\end{equation*}
Observe that $t_{i}(x) - \gamma$ is monic, has integer coefficients, and has linear coefficient $1$ (and constant term $0$). Thus its NP consists of a single horizontal line segment from $(1,0)$ to $(2^{i-1},0)$. From our inductive hypothesis, it follows that the NP of $2(t_{i}(x)-\gamma)t'_{i}(x)$ consists of a horizontal line segment from 
$(1,e)$ to $(2^{i-1},e)$, followed by a sequence of segments of slope $e/2^{i-2}, e/2^{i-3}, \ldots, e$ and respective lengths $2^{i-2}, 2^{i-3}, \ldots, 1$.  Hence the NP of $2(t_{i}(x)-\gamma)t'_{i}(x) + 1$ consists of a line segment from $(0,1)$ to $(2^{i-1},e)$, having slope $e/2^{i-1}$, and otherwise is identical to the NP of $2(t_{i}(x)-\gamma)t'_{i}(x)$, since $e/2^{i-1} < e/2^c$ for $0 \leq c \leq i-2$. This proves the claim. 

For each $i \ge 1$, $t_i(x)$ is a monic polynomial with degree $2^{i-1}$ and constant term $\gamma$, whose non-constant coefficients are all integers. If $v_\p(\gamma) \geq 0$, then the NP of $t_i(x)$ consists of nonpositive slopes, and hence all its roots have nonnegative $\p$-adic valuation, and therefore cannot coincide with roots of $t_i'(x)$ by the above claim. 
If $v_\p(\gamma) = s < 0$, the NP for $t_i(x)$ consists of a single line segment from $(0, s)$ to $(2^{i-1},0)$, with length $2^{i-1}$ and slope $-s/2^{i-1}$. Hence if $t_i(x)$ and $t_i'(x)$ have a root in common, then by the above claim, $-s/2^{i-1} = e/2^r$ with $0 \leq r \leq i-2$. But this holds if and only if $s = -e2^{i-1-r}$, and since $i-1-r \geq 1$, the proof is complete.   
\end{proof}

% By \cite[Theorem 5.11]{ADS}, $t_i(x)$ has exactly $2^{i-1}$ roots with $\p$-adic absolute value $\p^{1/2^{i-1}}$.  By contrast, $t_i'(x)$ is a polynomial of degree $2^{i-1}-1$ with integer coefficients, so its Newton polygon has a smaller total horizontal distance of $2^{i-1}-1$, and its vertices have integer coordinates.  It cannot have a line segment of slope $1/2^{i-1}$, hence $t_i'(x)$ cannot have a root of $\p$-adic absolute value $\p^{1/2^{i-1}}$.

%If $v_\p(\gamma)=0$, then the Newton polygon of $t_i(x)$ consists of a single horizontal line segment of length $2^{i-1}$ (from $(0,0)$ to $(2^{i-1},0)$), while if $v_\p(\gamma) > 0$, its Newton polygon consists of one segment of length 1 and negative slope (from $(0,v_\p(\gamma)$ to $(1,0)$) and one horizontal segment of length $2^{i-1}-1$ (from $(1,0)$ to $(2^{i-1},0)$).  Since  by the chain rule, all coefficients of $t_i'(x)$ have positive $\p$-adic valuation, since $\p \mid (2)$. Thus the Newton polygon for $t_i'(x)$ begins at $(0,0)$ and all of its segments have positive slope.  Again using  we see that if $v_\p(\gamma)\geq 0$, no roots of $t_i(x)$ can have the same $\p$-adic valuation as any root of $t_i'(x)$.  Hence $t_i(x)$ is separable in this case  as well.
%\end{proof}

\begin{remark}
To show that the genus of the curve $y^2 = t_i(x)$ is at least two, we can get by with a much weaker statement than Proposition \ref{twizzler}. Indeed, the genus of $y^2 = t_i(x)$ depends on the degree of $t_i(x)/f(x)$, where $f(x)$ is the largest square polynomial dividing $t_i(x)$. All we need to show is that the degree of $t_i(x)/f(x)$ is at least five, for each $i \geq 4$. It is reasonable to believe that this holds for all $\gamma \in \overline{\Q}$. 
\end{remark}

\section{The $n=2$ Case} \label{twocase}

Consider now the case where $n=2$. From \eqref{tigenus}, we know that when $t_3(x)$ is separable, $g_3=1$, and so $y^2 = t_3(x)$ is an elliptic curve. (When $t_3(x)$ is not separable, $y^2 = t_3(x)$ gives a curve of genus $0$.) Thus we cannot directly apply Faltings' Theorem, and we must use a different approach to determine the set $S(2, 2, \gamma)$ of $m\in K$ such that $g_{\gamma,m}^2(x)$ is irreducible and $g_{\gamma,m}^3(x)$ is reducible over $K$.

Now for some number fields $K$ and some $\gamma \in K$, it may still be the case that $y^2 = t_3(x)$ has only finitely many $K$-rational points, proving the finiteness of $S(2,2,\gamma)$ over $K$. This is the case for $\gamma = 0$ and $K = \Q$, as we now show:
\begin{prop} \label{deep dish pizza} Let $\gamma = 0$ and $C_3$ be the curve given by $y^2=t_3(x)=x^4-2x^3+x^2-x$. The only $\mathbb{Q}$-rational points on $C_3$ are $(0,0)$ and the point at infinity. In particular, there are no $m \in \Q$ such that $x^2 + m$ has a newly reducible third iterate. \end{prop}
\begin{proof}
Let $y=u/v^2$ and $x=-1/v$ define a birational rational map $\phi$ from $C_3': u^2=v^3+v^2+2v+1$ to $C_3$. We compute the conductor of the elliptic curve $C_3'$ to be 92, and locate it as curve 92A1 in Cremona's tables \cite{Cremona}. From the tables, it has rank zero over $\Q$ and torsion subgroup of order 3. Hence the obvious points $(0, \pm 1)$ together with the point at infinity give all $\Q$-rational points on $C_3'$. 
However, if $(x,y)$ is an affine rational point on $C_3$ with $x\not=0$, then $\phi^{-1}(x,y)$ is an affine rational point $(v,u)$ on $C_3'$ with $v\not=0$. But there are no such points.
\end{proof}

The strategy of Proposition \ref{deep dish pizza}, however, won't even work for all number fields $K$ in the case $\gamma = 0$. Indeed, let $K=\mathbb{Q}(i)$ and let $\phi$ be the same transformation as in Proposition \ref{deep dish pizza}. One can check that $(-1, i)$ is a non-torsion point of $C_3'$ in many ways. One of the more interesting, if not the simplest computationally, is to show that $(-1,i)$ has positive canonical height. In \cite{heightdifference}, Silverman gives upper and lower bounds for the difference between the canonical height $\hat{h}(P)$ and the Weil height $h(P)$ of a $K$-rational point $P$ on an elliptic curve, computed in terms of the discriminant and $j$-invariant of the curve.  For $C_3^\prime$, we have $-1.5484 \leq \hat{h}(P) - h(P) \leq 1.4577.$ In particular,  $\hat{h}(P) \geq h(P) - 1.5484$, so $h(P) > 1.5484$ would imply that $P$ is a non-torsion point.  Using MAGMA \cite{magma}, we find that although $h(P) = 0$ for $P = (-1,i)$ on $C_3'$, we have $h([2]P) = 1.6094$.  Thus $\hat{h}(P) = \frac{1}{4}\hat{h}([2]P) > 0$, using algebraic properties of canonical height.   

Since $(-1,i)$ is a non-torsion point, our curve $C_3$ has infinitely many $K$-rational points. However, when we check some corresponding $x$-values on $C_3$ as our choices for $m$ in $x^2 + m$, we don't find a newly reducible third iterate over $\mathbb{Q}(i)$. Thus we must adopt a different approach to have any hope of proving the $n=2$ case of Theorem $\ref{smores bites}$, even for $\gamma = 0$.

Let $K$ be a number field and $\gamma \in K$. Suppose that $g_{\gamma,m}^3(x)$ is newly reducible, so that by Corollary \ref{pulled pork}, $g_{\gamma,m}^3(x)=p_1(x)p_2(x)$ for irreducible polynomials $p_1(x), p_2(x) \in K[x]$ with $\deg p_1(x)=\deg p_2(x)=4$. Put $$p_1(x)=(x-\gamma)^4+a_3(x-\gamma)^3+a_2(x-\gamma)^2+a_1(x-\gamma)+a_0$$ and $$p_2(x)=(x-\gamma)^4+b_3(x-\gamma)^3+b_2(x-\gamma)^2+b_1(x-\gamma)+b_0$$  with $a_i, b_i \in K$. We also have
\begin{eqnarray}
g_{\gamma,m}^3(x)& = & (x-\gamma)^8+4m(x-\gamma)^6+(6m^2+2m)(x-\gamma)^4 \nonumber\\
& & +(4m^3+4m^2)(x-\gamma)^2+m^4+2m^3+m^2+m+\gamma. \nonumber
\end{eqnarray}
Multiplying $p_1(x)$ and $p_2(x)$ together, setting this product equal to $g_{\gamma,m}^3(x)$ and comparing coefficients we obtain a system of eight equations. By simplifying this system using Theorem \ref{baked ziti}, and noting that $a_0 \neq 0$ by the irreducibility of $p_1(x)$, we get two cases:

\vspace{5mm}

\noindent $\textbf{Case I:} \; a_1\not=0$, which implies $b_1=-a_1, b_2=a_2$:
\begin{enumerate}
\item $2a_2-a_3^2-4m=0$
\item $2a_0+a_2^2-2a_1a_3-6m^2-2m=0$
\item $2a_2a_0-a_1^2-4m^3-4m^2=0$
\item $a_0^2-m^4-2m^2-m^2-m-\gamma=0$
\end{enumerate}

\noindent $\textbf{Case II :} \; a_1=b_1=0$: 
\begin{enumerate}
\item $b_2-a_3^2+a_2-4m=0$
\item $(b_2-a_2)a_3=0$
\item $2a_0+a_2b_2-6m^2-2m=0 $
\item $(a_2+b_2)a_0-4m^3-4m^2=0$
\item $a_0^2-m^4-2m^2-m^2-m-\gamma=0.$
\end{enumerate}
%A system of five equations in $a_0$, $a_2$, $a_3$, $b_2$, and $m$. [Upon further reflection, we really should put these equations in].
\vspace{5 mm}

We use Groebner bases to find the solutions to these systems of non-linear equations. We dispense with Case II first, noting that it consists of five equations in five variables so we expect it will have only finitely many solutions in $\overline{K}$. We assign an ordering to the variables in which $\gamma$ is last, and using MAGMA \cite{magma} to compute a Groebner basis for each system, we find that the system in Case II has one $K$-rational solution for each $m \in K$ with  
\begin{eqnarray}
0 & = & m^{14} + m^{13}\gamma + \frac{13}{3}m^{13} + \frac{13}{3}m^{12}\gamma + \frac{22}{3}m^{12} + \frac{22}{3}m^{11}\gamma + \frac{57}{8}m^{11} + \frac{33}{4}m^{10}\gamma \nonumber\\
& & + 5m^{10} + \frac{9}{8}m^9\gamma^2 + \frac{23}{3}m^9\gamma + \frac{9}{4}m^9 + \frac{8}{3}m^8\gamma^2 + \frac{25}{6}m^8\gamma + \frac{7}{12}m^8 + 
        \frac{23}{12}m^7\gamma^2 \nonumber\\
& & + \frac{17}{12}m^7\gamma - \frac{1}{24}m^7 + \frac{13}{12}m^6\gamma^2 - \frac{1}{12}m^6\gamma- \frac{1}{12}m^6 + \frac{1}{4}m^5\gamma^3 - \frac{1}{24}m^5\gamma^2 \nonumber\\
& & - \frac{1}{4}m^5\gamma - \frac{1}{24}m^5 - \frac{1}{4}m^4\gamma^2 - \frac{1}{6}m^4\gamma - \frac{1}{12}m^3\gamma^3 - 
        \frac{1}{4}m^3\gamma^2 - \frac{1}{6}m^2\gamma^3 - \frac{1}{24}m\gamma^4.\nonumber
\end{eqnarray}
Clearly for any $\gamma \in K$, there are at most 14 such $m$, and so case II does not affect the finiteness of the number of $m$ for which $g_{\gamma,m}(x)$ has a newly irreducible third iterate.

Case I proves more interesting. We compute that that for fixed $\gamma \in K$, Case I has precisely one solution $(a_0, a_1, a_2, a_3, m) \in K^5$ for each $K$-rational point $(a_3, m)$ on the curve
\begin{eqnarray}
C_\gamma:0& = & a_3^{16} + 32a_3^{14}m + 352a_3^{12}m^2 - 32a_3^{12}m + 1792a_3^{10}m^3 - 256a_3^{10}m^2 \nonumber\\
& & + 4352a_3^8m^4 - 1536a_3^8m^3 - 1792a_3^8m^2 - 2176a_3^8m - 2176a_3^8\gamma \nonumber\\
& & + 4096a_3^6m^5 - 
        8192a_3^6m^4 - 12288a_3^6m^3 - 10240a_3^6m^2 - 10240a_3^6m\gamma \nonumber\\
& & - 16384a_3^4m^5 - 32768a_3^4m^4 - 38912a_3^4m^3 - 22528a_3^4m^2\gamma - 14336a_3^4m^2 \nonumber\\
& &  - 14336a_3^4m\gamma - 
        16384a_3^2m^4 - 16384a_3^2m^3\gamma - 16384a_3^2m^3 - 16384a_3^2m^2\gamma \nonumber\\
& & + 4096m^2 + 8192m\gamma + 4096\gamma^2. \nonumber
\end{eqnarray}
For instance, when $\gamma = 1/2$, one checks that $C_\gamma$ has the rational point $(1, 7/4)$, which corresponds to the newly reducible example given in \eqref{n3example}.  The actual Groebner basis is far too long to include here; however, we have included the Groebner basis in the case $\gamma = 1$ in the appendix to this article. Thus when $C_\gamma$ has genus at least two, there can be only finitely many $K$-rational solutions to the system given in Case I, and hence only finitely many $m \in K$ such that $g_{\gamma,m}(x)$ has a newly irreducible third iterate. Part (2) of Theorem \ref{smores bites} is thus proved when the genus $C_\gamma$ is at least two (Case II results in at most finitely many additional $m$-values, as will be shown below). 

Using MAGMA again, we checked that $C_\gamma$ has genus 11 for $\gamma = r/4, -200 \leq r \leq 200$ except for the following:
$$g(C_{-2}) = 9, \qquad g(C_0) = 9, \qquad g(C_{1/4}) = 7, \qquad g(C_1) = 10.$$
Note that we chose $\gamma$ to have denominator $4$ in order to include the case $\gamma = 1/4$, where we strongly suspected degeneracies to occur. 
The map $\psi$ sending $C_\gamma$ to $\gamma$ has fibers whose genus appears generally to be 11. Even the degenerate fibers seem to have genus greater than 1, and hence part (2) of Theorem \ref{smores bites} holds even in those cases. Interestingly, if we take a section of $\psi$ by fixing a value of $m$ and letting $\gamma$ vary, we appear always to get a curve of genus at most 1. This phenomenon was first noticed by Michael Zieve (personal correspondence). In other words, writing $C_{\gamma,m}$ instead of $C_\gamma$, and choosing $\psi'$ to be the map sending $C_{\gamma,m}$ to $m$, the surface $C_{\gamma,m}$ is (birational to) an elliptic surface. This observation may pave the way for a full understanding of $C_{\gamma, m}$, and hence improvements to part (2) of Theorem \ref{smores bites}.

%Thus there is a variety $V$ such that the map $\psi : C_\gamma \to V$ given by specializing $\gamma$ has fibers whose genus is generally 11. Even when the genus of a fiber unexpectedly goes down, it appears to remain greater than 1, and hence part (2) of Theorem \ref{smores bites} holds. Interestingly, if we take a section of $\psi$ by fixing a value of $m$ and letting $\gamma$ vary, we appear always to get a curve of genus at most 1. This phenomenon was first noticed by Michael Zieve (personal correspondence). In other words, choosing $\psi' : C_\gamma \to V'$ to be the map given by specializing $m$, the surface $C_\gamma$ is (birational to) an elliptic surface. 

% This proves part (2) of Theorem \ref{smores bites}.

\section*{Acknowledgements}
The authors are grateful to Michael Zieve for the suggestion of the terminology ``newly reducible," and for providing useful comments and computations. 

\bibliographystyle{plain}

\section*{Appendix}
\appendix
The resulting Groebner basis for Case I with $\gamma=1$ as calculated by MAGMA \cite{magma} is:
\begin{enumerate}
\item ${a_0} - {a_1}{a_3} + \frac{1}{8}{a_3}^4 - {a_3}^2q - q^2 + q$
\item ${a_1}^2 - {a_1}{a_3}^3 + 4{a_1}{a_3}q + \frac{1}{8}{a_3}^6 - \frac{3}{2}{a_3}^4q + 3{a_3}^2q^2 + {a_3}^2q$
\item ${a_1}{a_3}^5 + \frac{1920}{571}{a_1}{a_3}q^6 - \frac{35582}{1713}{a_1}{a_3}q^5 + \frac{641146}{15417}{a_1}{a_3}q^4 -
\frac{173966}{5139}{a_1}{a_3}q^3 + \frac{254212}{15417}{a_1}{a_3}q^2 - \frac{4322}{571}{a_1}{a_3}q +
\frac{35}{30834}{a_3}^{14}q - \frac{1}{1152}{a_3}^14 - \frac{4265}{123336}{a_3}^{12}q^2 +
\frac{200467}{7893504}{a_3}^{12}q + \frac{4199}{2631168}{a_3}^{12} +\frac{1775}{5139}{a_3}^{10}q^3 -
\frac{191455}{986688}{a_3}^{10}q^2 - \frac{75881}{986688}{a_3}^{10}q - \frac{22705}{15417}{a_3}^8q^4 +
\frac{516139}{986688}{a_3}^8q^3 + \frac{315853}{493344}{a_3}^8q^2 + \frac{54587}{986688}{a_3}^8q -
\frac{7}{48}{a_3}^8 + \frac{36880}{15417}{a_3}^6q^5 + \frac{76901}{61668}{a_3}^6q^4 -
\frac{148475}{30834}{a_3}^6q^3 + \frac{219505}{61668}{a_3}^6q^2 - \frac{11}{18}{a_3}^6q -
\frac{240}{571}{a_3}^4q^6 - \frac{429961}{61668}{a_3}^4q^5 + \frac{677423}{61668}{a_3}^4q^4 -
\frac{402371}{61668}{a_3}^4q^3 - \frac{75667}{123336}{a_3}^4q^2 + \frac{131047}{41112}{a_3}^4q +
\frac{1920}{571}{a_3}^2q^7 - \frac{35582}{1713}{a_3}^2q^6 + \frac{641146}{15417}{a_3}^2q^5 -
\frac{374378}{15417}{a_3}^2q^4 + \frac{152233}{15417}{a_3}^2q^3 - \frac{189763}{15417}{a_3}^2q^2 +
\frac{960}{571}q^5 - \frac{14911}{1713}q^4 + \frac{186374}{15417}q^3 - \frac{104975}{15417}q^2 +
\frac{4}{3}q$
\item ${a_1}{a_3}^2q + \frac{720}{571}{a_1}q^6 - \frac{17791}{2284}{a_1}q^5 + \frac{320573}{20556}{a_1}q^4 -
\frac{86983}{6852}{a_1}q^3 + \frac{53275}{10278}{a_1}q^2 - \frac{4199}{2284}{a_1}q -
\frac{45}{292352}{a_3}^{15}q^3 + \frac{14911}{18710528}{a_3}^{15}q^2 - \frac{93187}{84197376}{a_3}^{15}q +
\frac{104975}{168394752}{a_3}^{15} + \frac{45}{9136}{a_3}^{13}q^4 - \frac{14911}{584704}{a_3}^{13}q^3 +
\frac{93187}{2631168}{a_3}^{13}q^2 - \frac{11415}{584704}{a_3}^{13}q - \frac{1}{3072}{a_3}^{13} -
\frac{495}{9136}{a_3}^{11}q^5 + \frac{161141}{584704}{a_3}^{11}q^4 - \frac{1915915}{5262336}{a_3}^{11}q^3 +
\frac{300037}{1754112}{a_3}^{11}q^2 + \frac{206789}{7016448}{a_3}^{11}q + \frac{4199}{7016448}{a_3}^{11} +
\frac{315}{1142}{a_3}^9q^6 - \frac{101497}{73088}{a_3}^9q^5 + \frac{1170419}{657792}{a_3}^9q^4 -
\frac{154417}{219264}{a_3}^9q^3 - \frac{203785}{877056}{a_3}^9q^2 - \frac{75881}{2631168}{a_3}^9q -
\frac{765}{1142}{a_3}^7q^7 + \frac{236207}{73088}{a_3}^7q^6 - \frac{545431}{164448}{a_3}^7q^5 -
\frac{142777}{109632}{a_3}^7q^4 + \frac{4272259}{877056}{a_3}^7q^3 - \frac{4322155}{1315584}{a_3}^7q^2
+ \frac{3623737}{2631168}{a_3}^7q + \frac{360}{571}{a_3}^5q^8 - \frac{9151}{4568}{a_3}^5q^7 -
\frac{19973}{5139}{a_3}^5q^6 + \frac{128675}{6852}{a_3}^5q^5 - \frac{4341377}{164448}{a_3}^5q^4 +
\frac{1413245}{82224}{a_3}^5q^3 - \frac{830245}{164448}{a_3}^5q^2 - \frac{41}{48}{a_3}^5q -
\frac{1440}{571}{a_3}^3q^8 + \frac{20671}{1142}{a_3}^3q^7 - \frac{258976}{5139}{a_3}^3q^6 +
\frac{12676049}{164448}{a_3}^3q^5 - \frac{3880925}{54816}{a_3}^3q^4 + \frac{688435}{18272}{a_3}^3q^3 -
\frac{881653}{109632}{a_3}^3q^2 + \frac{172159}{109632}{a_3}^3q + \frac{2160}{571}{a_3}q^7 -
\frac{53373}{2284}{a_3}q^6 + \frac{320573}{6852}{a_3}q^5 - \frac{85543}{2284}{a_3}q^4 +
\frac{177007}{13704}{a_3}q^3 - \frac{148651}{41112}{a_3}q^2$
\item ${a_1}q^7 - \frac{68}{9}{a_1}q^6 + \frac{1606}{81}{a_1}q^5 - \frac{578}{27}{a_1}q^4 + \frac{853}{81}{a_1}q^3 -
\frac{50}{9}{a_1}q^2 + {a_1}q - \frac{1}{8192}{a_3}^{15}q^4 + \frac{59}{73728}{a_3}^{15}q^3 -
\frac{1075}{663552}{a_3}^{15}q^2 + \frac{377}{331776}{a_3}^{15}q - \frac{25}{73728}{a_3}^{15} +
\frac{1}{256}{a_3}^{13}q^5 - \frac{59}{2304}{a_3}^{13}q^4 + \frac{1075}{20736}{a_3}^{13}q^3 -
\frac{83}{2304}{a_3}^{13}q^2 + \frac{35}{3456}{a_3}^{13}q - \frac{11}{256}{a_3}^{11}q^6 + \frac{5}{18}{a_3}^{11}q^5 -
\frac{5647}{10368}{a_3}^{11}q^4 + \frac{9341}{27648}{a_3}^{11}q^3 - \frac{847}{13824}{a_3}^{11}q^2 -
\frac{275}{27648}{a_3}^{11}q - \frac{1}{3072}{a_3}^{11} + \frac{7}{32}{a_3}^9q^7 - \frac{101}{72}{a_3}^9q^6 +
\frac{3497}{1296}{a_3}^9q^5 - \frac{5249}{3456}{a_3}^9q^4 + \frac{203}{1728}{a_3}^9q^3 +
\frac{365}{10368}{a_3}^9q^2 + \frac{11}{1152}{a_3}^9q - \frac{17}{32}{a_3}^7q^8 + \frac{949}{288}{a_3}^7q^7 -
\frac{7261}{1296}{a_3}^7q^6 + \frac{1103}{3456}{a_3}^7q^5 + \frac{19607}{3456}{a_3}^7q^4 -
\frac{58525}{10368}{a_3}^7q^3 + \frac{15673}{5184}{a_3}^7q^2 - \frac{863}{1152}{a_3}^7q + \frac{1}{2}{a_3}^5q^9
- \frac{41}{18}{a_3}^5q^8 - \frac{115}{81}{a_3}^5q^7 + \frac{4409}{216}{a_3}^5q^6 - \frac{23737}{648}{a_3}^5q^5
+ \frac{19853}{648}{a_3}^5q^4 - \frac{2225}{162}{a_3}^5q^3 + \frac{427}{216}{a_3}^5q^2 - 2{a_3}^3q^9 +
\frac{154}{9}{a_3}^3q^8 - \frac{37351}{648}{a_3}^3q^7 + \frac{8318}{81}{a_3}^3q^6 - \frac{11993}{108}{a_3}^3q^5
+ \frac{3571}{48}{a_3}^3q^4 - \frac{776}{27}{a_3}^3q^3 + \frac{3539}{432}{a_3}^3q^2 - \frac{41}{48}{a_3}^3q +
3{a_3}q^8 - \frac{68}{3}{a_3}q^7 + \frac{1606}{27}{a_3}q^6 - \frac{1147}{18}{a_3}q^5 + \frac{778}{27}{a_3}q^4 -
\frac{2075}{162}{a_3}q^3 + \frac{49}{18}{a_3}q^2$
\item${a_2} - \frac{1}{2}{a_3}^2 + 2q$
\item${a_3}^{16} - 32{a_3}^{14}q + 352{a_3}^{12}q^2 + 32{a_3}^{12}q - 1792{a_3}^{10}q^3 - 256{a_3}^{10}q^2 + 4352{a_3}^8q^4 + 1536{a_3}^8q^3 - 1792{a_3}^8q^2 + 2176{a_3}^8q - 4096{a_3}^6q^5 - 8192{a_3}^6q^4 + 12288{a_3}^6q^3 - 
10240{a_3}^6q^2 + 16384{a_3}^4q^5 - 32768{a_3}^4q^4 + 38912{a_3}^4q^3 - 14336{a_3}^4q^2 - 16384{a_3}^2q^4 + 
16384{a_3}^2q^3 + 4096q^2$
\end{enumerate}

\end{document}